\documentclass[11pt]{amsart}

\usepackage{amsmath}
\usepackage{amsfonts}
\usepackage{amssymb}
\usepackage{amsthm}
\usepackage{indentfirst}
\usepackage[OT2,T1]{fontenc}
\usepackage{lmodern}

\usepackage{verbatim}

\usepackage{enumerate}

\newtheorem{theorem}{Theorem}[section]
\newtheorem{lemma}[theorem]{Lemma}
\newtheorem{proposition}[theorem]{Proposition}
\newtheorem{corollary}[theorem]{Corollary}

\newtheorem{definition}[theorem]{Definition}

\newtheorem{remark}[theorem]{Remark}
\newtheorem{notation}[theorem]{Notation}

\newcommand{\Gal}{\operatorname{Gal}}
\newcommand{\ds}{\displaystyle}
\newcommand{\Tr}{\operatorname{Tr}}

\newcommand\rank{\operatorname{rank}}

\newcommand\Qp{{\mathbb Q_p}}
\newcommand\Zp{{\mathbb Z_p}}

\newcommand\Z{\mathbb Z}

\newcommand\mfp{\mathfrak p}

\newcommand\barmfp{\bar{\mfp}}

\newcommand\Q{\mathbb Q}

\newcommand\C{\mathbb C}

\newcommand\OO{\mathcal O}

\newenvironment{nouppercase}{%
  \renewcommand{\uppercasenonmath}[1]{}}{}

\DeclareSymbolFont{cyrletters}{OT2}{wncyr}{m}{n}
\DeclareMathSymbol{\Sha}{\mathalpha}{cyrletters}{"58}

\newcommand\Label{\label}

\title[Euler Systems of $X_1(N)$]{Construction of Anti-Cyclotomic Euler Systems of Abelian Varieties Associated to $X_1(N)$}

\author{Daeyeol Jeon}
\address{Department of Mathematics Education, Kongju National University,
Gongju 314-701, South Korea}
\email{dyjeon@kongju.ac.kr}

\author{Byoung Du Kim}
\address{School of Mathematics and Statistics, Victoria University of Wellington, Wellington 6140, New Zealand}
\email{byoungdu.kim@vuw.ac.nz}

\author{Chang Heon Kim}
\address{Department of Mathematics, Sungkyunkwan University, Suwon 440-746,
South Korea}
\email{chhkim@skku.edu}


\begin{document}

\begin{nouppercase}
\maketitle
\end{nouppercase}

\begin{abstract}
Let $K$ be an imaginary quadratic field, $N$ be a positive integer, $f(z)$ be a newform of level $\Gamma_1(N)$, and $A_f$ be the abelian variety associated to $f$. For each $\tau \in K$ ($\operatorname{Im} \tau >0$), we construct a certain point $P_\tau$ on $A_f$ defined over an extended ring class field of $K$ of level $N$. Our construction generalizes Birch's construction of the Heegner points to the abelian varieties associated to modular forms of level $\Gamma_1(N)$ and nontrivial character. Then, we show that $P_\tau$'s satisfy the distribution and congruence relations of an Euler system, which implies that it should be possible to apply the Euler system techniques to them to show a relation between the non-torsionness of $P_\tau$ and the rank of $A_f(K)$.

\end{abstract}

\tableofcontents

\section{Introduction}		\label{Introduction}

In this paper, we present the construction of certain points on the modular curve $X_1(N)$ (and by extension, on the abelian varieties over $\Q$ given by irreducible quotients of $J_1(N)$). We will argue that our points generalize Birch's Heegner points (\cite{Birch-1}) (which are defined on $X_0(N)$) in the sense that (like Birch's construction) modular functions and class field theory play an integral role in the construction, and show that they satisfy the conditions (the distribution and congruence relations) of an Euler system.

First, we give a brief description of Birch's construction of the Heegner points. As is customary, we let $\mathbb H$ denote the upper-half plane, let $Y_0(N)$ (resp. $Y_1(N)$) denote $\Gamma_0(N) \backslash \mathbb H$ (resp. $\Gamma_1(N) \backslash \mathbb H$), and $X_0(N)$ (resp. $X_1(N)$) denote its compactification by the addition of cusps. Let $j(\tau)$ ($\tau \in \mathbb H$) denote the modular elliptic function given by the $j$-invariant of $\Lambda_\tau=(1, \tau)=\Z\cdot 1 + \Z \cdot \tau$. It is well-known that $(j(\tau),j(N\tau))$ satisfies a certain polynomial equation $P_N(X,Y)=0$, which gives an affine model over $\Q$ of $Y_0(N)$. If $K$ is an imaginary quadratic field, and $\tau \in K \cap \mathbb H$, then $j(\tau)$ generates a certain ring class field extension of $K$ by class field theory. Birch noted that for a carefully chosen $\tau$, $(j(\tau), j(N\tau))$ is a point on the affine model $Y_0(N)_{/\Q}$ over the ring class field (thus can be considered as a point on $X_0(N)_{/\Q}$), which he called a Heegner point (\cite{Birch-1}). A Heegner point can be also considered as a point on an elliptic curve $E$ over $\Q$ of conductor $N$, as we will explain shortly.

On the other hand, the authors wanted to find an explicit way to construct a point on $X_1(N)$ also defined over a certain ring class field extension of $K$. Noting the role played by $j(\tau)$ in Birch's construction, we looked for modular functions that can play a similar role.

In \cite{Baaziz}, Baaziz constructed modular functions $b(\tau), c(\tau)$ of level $\Gamma_1(N)$ (Section~\ref{sec:Preliminaries}), which are rational functions of the Weierstrass functions $\wp(\cdot; \Lambda_\tau), \wp'(\cdot; \Lambda_\tau)$, and generate the function field of $X_1(N)$. Jeon, Kim, and Lee noted (\cite{Jeon-Kim-Lee}) that $(b(\tau), c(\tau))$ gives an affine model (over $\Q$) of $Y_1(N)$. The modular functions $b(\tau), c(\tau)$ seemed ideal for our purpose.

Our first goal is to define a point analogous to the Heegner points: We define $P_{\tau}=(b(\tau), c(\tau)) \in Y_1(N) (\subset X_1(N))$ for any $\tau \in K$, $\operatorname{Im}\tau>0$.

Secondly, we find a number field over which $P_\tau$ is defined. For an order $\OO$ of $K$, let $L_{\OO,N}$ be the {\it extended ring class field of level $N$ associated to $\OO$} (see Section~\ref{Field}). We show that if $\OO$ acts on $\Lambda_\tau$, then $P_{\tau} \in X_1(N)(L_{\OO,N})$ (Corollary~\ref{Carp}).

Thirdly, we show that $P_\tau$ satisfies the distribution and congruence relations of Euler systems.

Let $f(z)=\sum_{n=1}^{\infty} a_n(f) q^n$ ($a_1=1$) be a newform of level $\Gamma_1(N)$ with character $\epsilon$ (modulo $N$), $A_f$ be the abelian variety given by the quotient of $J_1(N)$ divided by the ideal of the Hecke algebra generated by $T_l-a_l(f)$ ($l\nmid N$), $U_l-a_l(f)$ ($l|N$), $\langle l \rangle-\epsilon(l)$ ($(l,N)=1$) where $l$ runs over all primes, and $\mu_f: X_1(N)\to J_1(N) \to A_f$ be a modular parametrization (where the map from $X_1(N)$ to $J_1(N)$ is given by $P \mapsto (P)-(\infty)$). We let $P_{\tau}$ also denote $\mu_f(P_{\tau}) \in A_f$ by abuse of notation.

Suppose $K=\Q(\sqrt D)$ for some square-free negative integer $D$. Fix $\tau_K=\sqrt D$ if $D \not\equiv 1 \pmod 4$, and $\ds \frac{\sqrt D+1}2$ if $D \equiv 1 \pmod 4$. For each positive integer $c$ prime to $N$, let

\[ \tau'=\frac{a+\tau_K}c \]
for an integer $a \in \Z$. Then, as mentioned above, $P_{\tau'}$ is defined over $L_{\OO_c,N}$ where $\OO_c=\Z+c\OO_K$.

Suppose $p$ is a prime number prime to $N \cdot \operatorname{disc}(K/\Q)$, and let $a_p(f)$ be the $p$-th Fourier coefficient of the $q$-expansion of $f$.

If $(p, c)=1$, $p \equiv 1 \pmod N$, and $p$ is inert over $K/\Q$, then

\begin{eqnarray}		\label{JKK-1} 
\Tr_{ L_{\OO_{cp}, N} / L_{\OO_c, N} } P_{\tau'/p} = a_p(f) P_{\tau'}. 
\end{eqnarray}
(See Theorem~\ref{Spain}.)

On the other hand, if $p|c$, then

\begin{eqnarray}		\label{JKK-2}
\Tr_{ L_{\OO_{cp},N}/ L_{\OO_c, N} } P_{\tau'/p} = a_p(f) P_{\tau'}-\epsilon(p)P_{p\tau'}. 
\end{eqnarray}
(See Theorem~\ref{Spanish March}.)

Now, suppose $p$ is a prime that is inert over $K/\Q$ and $p \equiv 1 \pmod N$, and $(p,c)=1$. Let $\lambda$ be any prime of $L_{\OO_c, N}$ lying above $p$, $\lambda'$ be any prime of $L_{\OO_{cp}, N}$ lying above $\lambda$, and $\operatorname{red}_{\lambda}$ and $\operatorname{red}_{\lambda'}$ be reduction maps onto the special fiber of the N\'eron model (over $\Zp$) of $A_f$. Then, we have

\begin{align}		
\Label{JKK-3} (p+1)\operatorname{red}_{\lambda'} P_{\tau'/p}	&=& (\operatorname{Frob}_p + p\cdot \epsilon(p) \cdot \operatorname{Frob}_p^{-1}) \operatorname{red}_{\lambda} P_{\tau'}	\\
\nonumber &=&a_p(f) \cdot \operatorname{red}_{\lambda} P_{\tau'}
\end{align}
(see Theorem~\ref{Congruence}).

Now, let's compare them with the conditions of Kolyvagin's Euler system of the Heegner points (\cite{Kolyvagin-3}~Sections~1, 3). Let $C_\tau$ denote Birch's Heegner point $(j(\tau), j(N\tau)) \in X_0(N)$. Suppose $f(z)$ is a newform of level $\Gamma_0(N)$ and $A_f$ is the abelian variety associated to $f(z)$ as in the case of $X_1(N)$ (the most prominent case being an elliptic curve over $\Q$). Again, we fix a modular parametrization map $\mu_f: X_0(N) \to A_f$ defined over $\Q$ which satisfies $\mu_f(\infty)=0$. By abuse of notation, we let $C_\tau$ denote $\mu_f(C_\tau) \in A_f$ as well. By Kolyvagin (\cite{Kolyvagin-1}, \cite{Kolyvagin-2}), for each $n \in \Z (n>0$) we can choose an appropriate $\tau_n \in K \cap \mathbb H$ so that $C_{\tau_n}$ is defined over the ring class field $L_{\OO_n}$ of the order $\OO_n=\Z+n\OO_K$, and for a prime $l$ with $(l,N)=1$, if $l\nmid n$ and $l$ is inert over $K/\Q$, we have the distribution relation

\begin{eqnarray}	\label{Kolyvagin-1}
 \Tr_{L_{\OO_{nl}}/L_{\OO_n}} C_{\tau_{nl}}=a_l(f) C_{\tau_n} 
\end{eqnarray}
(\cite{Kolyvagin-1}~Proposition~1). Although, it does not appear in Kolyvagin's work, we also have that if $l | n$,

\begin{eqnarray}	\label{Kolyvagin-2}
\Tr_{L_{\OO_{nl}}/L_{\OO_n}} C_{\tau_{nl}}=a_l(f) C_{\tau_n} - C_{\tau_{n/l}} 
\end{eqnarray}
(see the proof of \cite{Rubin-4}~Proposition~6.1 although the readers should note that Rubin assumes $a_l(f)=0$). 

Also, where $l$ does not divide $\operatorname{disc}(K/\Q)$, $v$ is any prime of $K(1)$ above $l$, and $w$ is any prime of $L_{\OO_l}$ above $v$, we have the congruence relation

\begin{eqnarray}	\Label{Kolyvagin-3}
\operatorname{red}_w (C_{\tau_l})= \operatorname{Frob}_l (\operatorname{red}_v(C_{\tau_1}))
\end{eqnarray}
(see \cite{Kolyvagin-1}~Proposition~6, and \cite{Kolyvagin-2}~Proposition~1).

(\ref{JKK-1}) is clearly analogous to (\ref{Kolyvagin-1}) with the extra condition $p\equiv 1 \pmod N$, which we believe will not make much difference in practice. (\ref{JKK-2}) is also clearly analogous to (\ref{Kolyvagin-2}). The appearance of $\epsilon(p)$ can be easily explained by the fact that modular forms of level $\Gamma_0(N)$ have a trivial character. (\ref{JKK-2}) should be what Rubin calls the distribution relation in the $p$-direction (\cite{Rubin-3}~Remark~2.1.5), and (as Rubin points out) we believe that it can replace the congruence relations in the Euler system techniques. Also it should be noted that (\ref{JKK-2}) indicates a natural connection with Iwasawa Theory.

The main goal of an Euler system is to obtain a sharp bound for the ranks of $A_f$. For example, Kolyvagin showed that if $f$ is a newform of level $\Gamma_0(N)$ and $N_{K(1)/K}C_1 (\in A_f(K))$ is not torsion, then $\rank A_f(K)=1$. We believe that we can apply the techniques of Euler systems to $\{ P_{(a+\tau)/c} \}$, and obtain a similar result for $A_f(K)$ where $f$ is a newform of level $\Gamma_1(N)$ (and a more general result in the direction of Iwasawa Theory), and we are hopeful that such a result will be in our subsequent publication.

\begin{remark}
There are also Kato's Euler systems (\cite{Kato}) defined on $J(N)$. We note that his Euler systems are ``the Euler systems over the cyclotomic fields'' whereas our Euler system (as well as the Euler system of the Heegner points) are ``the Euler systems over anti-cyclotomic fields.'' They are different in the definition, construction, and application.
\end{remark}

%%%%%%%%%%%%%%%%%%%%%%%
%%%%%%%%%%%%%%%%%%%%%%%

\section{Preliminaries}\label{sec:Preliminaries}

\subsection{Modular functions $b(\tau)$ and $c(\tau)$}
\hfill

Let $\Gamma=SL_2(\Z)$ be the full modular group, and for any $N\geq 1$, $\Gamma(N)$, $\Gamma_1(N)$, and $\Gamma_0(N)$ be the standard congruence groups.
Let $Y_1 (N)_{/\Q}$ be the affine curve over $\Q$ of the moduli schemes of the isomorphism classes of elliptic curves $E$ with an $N$-torsion point. As well-known, $Y_1(N)_{\C}$ is (isomorphic to) $(Y_1 (N)_{/\Q} \otimes \C)^{an}$.

More explicitly, this isomorphism is given by the following: Let $\Lambda_\tau=(\tau,1)$ be the lattice in $\C$ with basis $\tau$ and 1. Then, the above-mentioned isomorphism (of analytic curves between $Y_1(N)_{\C}$ and $(Y_1 (N)_{/\Q} \otimes \C)^{an}$) is given by

\[ \tau \mapsto \left( 	\C/\Lambda_\tau,\frac{1}{N}+\Lambda_\tau		\right).	\]

%The $X$'s are compact Riemann surfaces. Denote the genera of $X_1(N), $X_0(N)$ by $g_1(N), g_0(N)$ respectively.

%Let $\HH$ be the complex upper half plane, and let
%$$\Gamma_1(N):=\left\{\begin{pmatrix} a&b\\c&d\end{pmatrix}\in \Gamma(1):=SL_2(\Z)\,|\, a\equiv d\equiv 1\,\,({\rm mod}\,\, N),\, c\equiv 0\,\,({\rm mod}\,\, N)\right\}.$$
%Then $\Gamma_1(N)$ acts on $\HH$ by linear fractional transformations, and the quotient space $Y_1(N)=\Gamma_1(N)\backslash\HH$ parametrizes the isomorphism classes of elliptic curves with prescribed $N$-torsion points.
%By adding finitely many points, called {\it cusps}, one can get its compactification which is denoted by $X_1(N)$.

%The modular curve $X_1(2,2N)$ corresponds to the congruence subgroup
%$$\Gamma_1(2,2N):=\left\{\begin{pmatrix} a&b\\c&d\end{pmatrix}\in \Gamma_1(2N)\,|\, b\equiv 0\,\,({\rm mod}\,\, 2)\right\}.$$
%% Fix a primitive $M$th root of unity $\zeta_M$. Let $e_M$ be a Wel pairing.

%We need some more modular curves. Let $\Delta$ be a subgroup of $({\mathbb Z}/{N \mathbb Z})^*$ which contains $-1.$ Let $X_\Delta(N)$ be the modular curve defined over $\mathbb Q$ associated to the congruence subgroup
%$$\Gamma_\Delta(N):=\left\{{\begin{pmatrix}a&b\\c&d\end{pmatrix}} \in\Gamma(1)\,|\,a\in\Delta, N\mid c \right\}.$$ Note that for
%$\Delta=\{\pm 1\}$ this is just $X_1(N).$

The {\it Tate normal form} of an elliptic curve with point $P=(0,0)$ is as follows:
$$E=E(b,c) : Y^2+(1-c)XY -bY=X^3 -bX^2,$$ and this is nonsingular if and only if $b\neq0.$
On the curve $E(b,c)$ we have the following by the chord-tangent method:
\begin{align}\label{eq:nP}\notag  \\ \notag
P&=(0,\ 0),\\ \notag
2P&=(b, \ bc),\\ \notag
3P &=(c, \ b-c),\\ 
4P &=\left(\frac{b(b-c)}{c^2}, \ -\frac{b^2(b-c-c^2)}{c^3}\right), \\ \notag
5P &= \left(-\frac{bc(b-c-c^2)}{(b-c)^2}, \ \frac{bc^2(b^2-bc-c^3)}{(b-c)^3}\right), \\ \notag
6P &= \left( \frac{(b-c)(b^2-bc-c^3)}{(b-c-c^2)^2}, \ \frac{c(2b^2-3bc-bc^2+c^2)(b-c)^2}{(b-c-c^2)^3}  \right).
\end{align}

In fact, the condition $NP= O$ in $E(b,c)$ gives a defining equation for $X_1(N)$.
For example, $11P=O$ implies $5P=-6P$, so
$$x_{5P}=x_{-6P}=x_{6P},$$
where $x_{nP}$ denotes the $x$-coordinate of the $n$-multiple $nP$ of $P$.
Eq. (\ref{eq:nP}) implies that
\begin{equation}\label{mod11}
-\frac{bc(b-c-c^2)}{(b-c)^2} = \frac{(b-c)(b^2-bc-c^3)}{(b-c-c^2)^2}.
\end{equation}
Without loss of generality, the cases $b=c$ and $b=c+c^2$ may be excluded.
Then Eq. (\ref{mod11}) becomes as follows:
\begin{align*}
-b^2c^3-6bc^5+3b^3c^2+9b^2c^4-3bc^6-3b^4c-4b^3c^3+3b^2c^5-bc^7+c^6+b^5=0,
\end{align*}
which is one of the equation $X_1(11)$ called the {\it raw form} of $X_1(11)$.
By the coordinate changes $b=(1-x)xy(1+xy)$ and $c=(1-x)xy$, we get the following equation:
$$f(x,y):=y^2+(x^2+1)y+x=0.$$

%\begin{align}\label{eq:nP}\notag
%P&=(0,\ 0),\\\notag
%2P&=(b, \ bc),\\\notag
%3P &=(c, \ b-c),\\
%4P &=\left(r(r-1), \ r^2(c-r+1)\right); \ \ b=cr,\\\notag
%5P &= \left(rs(s-1), \ rs^2(r-s)\right); \ \ c=s(r-1), \\ \notag
%6P &= \left( \frac{s(r-1)(r-s)}{(s-1)^2}, \ \frac{s^2(r-1)^2(rs-2r+1)}{(s-1)^3}  \right).
%\end{align}

%Very recently, by using the Tate normal form, Sutherland~\cite{Su} found optimized forms for defining equations of the modular curves $X_1(N)$ for $N=11, 13-50.$
%The formulas in Table~\ref{tb:Sutherland-table6} and Table~\ref{tb:Sutherland-table7} are taken directly from~\cite[Table 6 and Table 7]{Su}.
%We use those defining equations for $N=17,20,21,22,24$, which are given in Table~\ref{tb:Sutherland-table6}. We also need Table~\ref{tb:Sutherland-table7} for birational maps for $X_1(N)$ for our purpose.

%The condition $NP= O$ in $E(b,c)$ gives a defining equation for $X_1(N)$.
%For example, $11P=O$ implies $5P=-6P$, so
%$$x_{5P}=x_{-6P}=x_{6P},$$
%where $x_{nP}$ denotes the $x$-coordinate of $nP$ of $P$.
%Eq.~(\ref{eq:nP}) implies that
%\begin{equation}\label{mod11}
%rs(s-1) =   \frac{s(r-1)(r-s)}{(s-1)^2}.
%\end{equation}
%Without loss of generality, the cases $s=1$ and $s=0$ may be excluded.
%Then Eq.~(\ref{mod11}) becomes the following:
%$$r^2-4sr+3s^2r-s^3r+s=0,$$
%which is a defining equation for $X_1(11)$, called the {\it raw form} of $X_1(11)$. 
%By the coordinate changes $s=1-x$ and $r=1+xy$, we get another defining equation for $X_1(11)$ as follows:
%$$y^2+(x^2+1)y+x=0.$$

Now we note that
\begin{align*}
\left(	\C/\Lambda_\tau, \,  \frac{1}{N}+\Lambda_\tau	\right)		&=  \left(	y^2=4x^3-g_2(\tau)x-g_3(\tau),\, \left(\wp\left(\frac{1}{N}; \Lambda_\tau	\right),\wp'\left(\frac{1}{N}; \Lambda_\tau	\right)\right)\right) \\
&=\left(	y^2+(1-c(\tau))xy-b(\tau)y=x^3-b(\tau)x^2,\, (0,0)\right),
\end{align*}
where $\wp(z; \Lambda_\tau)$ is the Weierstrass elliptic function of the period $\Lambda_\tau$.
From \cite{Baaziz}, it follows that
\begin{equation}\label{eq:b-c}
b(\tau)=-\frac{(\wp(\frac{1}{N}; \Lambda_\tau	)-\wp(\frac{2}{N}; \Lambda_\tau	))^3}{\wp'(\frac{1}{N}; \Lambda_\tau	)^2},\,\,
c(\tau)=-\frac{\wp'(\frac{2}{N}; \Lambda_\tau	)}{\wp'(\frac{1}{N}; \Lambda_\tau	)}
\end{equation}
are modular functions on $\Gamma_1(N)$ and
generate the function field of $X_1(N)$, where the derivative $\wp'$ is with respect to $z$.
%Also the function field of $X_1(N)$ can be generated by $x,y$ satisfying the equation $f_N(x,y)$
%in Table \ref{tb:Sutherland-table6} for the case $N=17,20,21,22,24$, and $x,y$ are considered as functions of $\tau$ via the rational maps of (\ref{tb:Sutherland-table6}) and (\ref{eq:b-c}).

\subsection{Field of definitions of $b(\tau)$ and $c(\tau)$ for a CM-points $\tau$}			\Label{Field}

Let $\mathcal F_N$ be the extension of the function field $\Q(j(\tau))$ generated by the Fricke functions indexed by $r \in \frac 1N \Z^2/\Z^2$ (see \cite{Koo-Shin-Yoon}~Section~4), where $j(\tau)$ is the modular invariant function.
By the theory of modular functions, it is known that $\mathcal F_N$ is the set of all functions in $\C(X(N))$ whose Fourier coefficients are in $\Q(\zeta_N)$, $\mathcal F_1$ is simply $\Q(j(\tau))$, and

\[ \Gal(\mathcal F_N/\mathcal F_1) \cong \operatorname{GL}_2(\Z/N\Z) / \{ \pm I_2 \} \cong G_N \cdot \operatorname{SL}_2(\Z/N\Z) / \{ \pm I_2 \} \]
where

\[ G_N = \left\{ \left. \begin{bmatrix} 1&0 \\ 0&d \end{bmatrix} \; \right| \; d \in (\Z/N\Z)^* \right\}.	\]

The functions $b(\tau), c(\tau)$ have their Fourier coefficients in $\Q(\zeta_N)$, and they are contained in $\mathcal F_N$. 

\vspace{2mm}

\begin{definition}		\Label{Rome}
\[ P_\tau=(b(\tau),c(\tau)) \in X_1(N).	\]
\end{definition}

\vspace{2mm}

Let $\mathcal O$ be an order of conductor $c$ in an imaginary quadratic field $K$. The ring class field of $\mathcal O$, denoted by $L_{\mathcal O}$, is determined via the Existence Theorem of class field theory \cite[Theorem 8.6]{Cox} by the subgroup 
$P_{K, \mathbb Z}(c) \subset I_K(c)$ generated by principal ideals $\alpha \mathcal O_K\in I_K(c)$ where $\alpha\equiv a \mod c\mathcal O_K$ for some $a\in \mathbb Z$. 
Here $I_K(c)$ denotes the group of all fractional ideals relatively prime to $c$.
This implies that 
$$ 
Gal (L_{\mathcal O}/K) \cong I_K(c)/P_{K,\mathbb Z}(c) \cong C(\mathcal O), 
$$
where $C(\mathcal O)$ is the class group of $\mathcal O$. 
Following \cite{Cho} we define 
$$
P_{K, \mathbb Z, N}(cN) \subset I_K(cN)
$$
to be the subgroup generated by the principal ideals $\alpha \mathcal O_K\in I_K(cN)$ where
$\alpha\in \mathcal O_K$ satisfies 
$$
\alpha\equiv a \mod cN\mathcal O_K
\hbox{ for some $a\in \mathbb Z$ with $a\equiv 1 \mod N$}. 
$$
It then follows from the Existence Theorem that there exists an extension $L_{\mathcal O,N}$ called the {\it extended ring class field of level $N$}, with Galois group
$$
Gal (L_{\mathcal O,N}/K) \cong I_K(cN)/P_{K,\mathbb Z,N}(cN).
$$
We note that $L_{\mathcal O, 1}=L_{\mathcal O}$ and $L_{\mathcal O, N}$ is a Galois extension of $L_{\mathcal O}$. In particular, if $\mathcal O=\mathcal O_K$, then $L_{\mathcal O, N}$ is equal to the ray class field $K(N)$. 

A point $\tau \in K\cap \mathbb H$ is a root of $ax^2+bx+c$ where $a,b, c\in \mathbb Z$ are relatively prime with $a>0$. Then the lattice $L_\tau=[1,\tau]$ is a proper ideal for the order $\mathcal O=[1, a\tau]$ (see \cite[Theorem 7.7]{Cox}). As a consequence of Shimura reciprocity we have the following theorem.

\begin{theorem} \cite[Theorem 15.16]{Cox}	\Label{Cox}
Fix $\tau \in K\cap \mathbb H$ and 
$\mathcal O$ as above and assume that $f(\tau)$ is well-defined for a modular function $f\in \mathcal F_N$. Then $f(\tau)\in L_{\mathcal O, N}$. 
\end{theorem} 
Thus we have the following immediate corollary.

\begin{corollary}			\Label{Carp}
Fix $\tau \in K\cap \mathbb H$ and 
$\mathcal O$ as above and assume that $b(\tau), c(\tau)$ are defined. Then  $b(\tau), c(\tau)\in L_{\mathcal O, N}$ and
therefore the point $P_\tau$ is defined over
$L_{\mathcal O, N}$.
\end{corollary}

\begin{proof} This immediately follows from Theorem~\ref{Cox} because $b(\tau),c(\tau) \in \mathcal F_N$.
\end{proof}

%%%%%%%%%%%%%%%%%%%

\vspace{2mm}

\begin{section}{The Main Theorem of Complex Multiplication, and the action of the Galois groups on $P_\tau$} \label{Berlin}
\vspace{2mm}

In this section, we apply Shimura's theory of complex multiplication to $P_{\tau}$ to study the action of the Galois groups of extended ring class fields, and in particular, we find the field of definition of $P_{\tau}$ by other means.

As before, the lattice $(\alpha, \alpha')$ denotes $\Z\alpha+\Z\alpha'$.

%%%%%%%%%%%%%%%%%%%%%%%%%
%%%%%%%%%%%%%%%%%%%%%%%%%
%%%%%%%%%%%%%%%%%%%%%%%%%
%%%%%%%%%%%%%%%%%%%%%%%%%
%%%%%%%%%%%%%%%%%%%%%%%%%
%%%%%%%%%%%%%%%%%%%%%%%%%
%%%%%%%%%%%%%%%%%%%%%%%%%

The following is from \cite{Shimura}~Section~5.2 and Section~5.3.

Suppose $\Lambda$ is an arbitrary $\Z$-lattice in $K$. For each rational prime $p$, let $K_p=K\otimes_\Q \Qp$ and $\Lambda_p = \Lambda \otimes_\Z \Zp$ (so that $\mathbb A_K=\prod_p K_p$). It is worth noting that if $p$ splits completely over $K/\Q$ (so that $p\OO_K=\mfp\barmfp$), then $K_p=K_\mfp \times K_{\barmfp}$.

For any $x \in \mathbb A_K^*$, we may speak of the $p$-component $x_p$ of $x$ belonging to $K_p^*$. (In other words, if $p$ is inert, $x_p \in K_{p\OO_K}^*$, if $p$ splits completely, $x_p=(x_\mfp, x_{\barmfp}) \in K_\mfp^* \times K_{\barmfp}^*$, and if $p$ is ramified, $x_p \in K_\mfp^*$ for the unique prime $\mfp$ of $\OO_K$ above $p$.)

We observe that $x_p \Lambda_p$ is a $\Zp$-lattice in $K_p$. It is well-known that there exists a $\Z$-lattice $\Lambda'$ in $K$ such that $\Lambda'_p=x_p \Lambda_p$ for every $p$ (\cite{Shimura}~page 116). Then, we define

\[ x\Lambda \stackrel{def}= \Lambda'. \]

The isomorphism $x: K/\Lambda \stackrel{\times x}\to K/x \Lambda$ is given as follows: Since $\Q/\Z =\prod_p \Qp/\Zp$ canonically, we have the canonical decomposition $K/\Lambda = \prod_p K_p/ \Lambda_p$. There is a well-defined isomorphism given by multiplication $x_p: K_p/ \Lambda_p \stackrel{\times x_p}\to K_p/ x_p \Lambda_p$ for each prime $p$. Combining them for all $p$, we obtain an isomorphism $x: K/\Lambda \to K/x \Lambda$. In other words, $x: K/\Lambda \to K/x \Lambda$ is an isomorphism which makes the following diagram commutative for every prime $p$:

\begin{eqnarray*}
K_p/\Lambda_p	&	\stackrel{x_p}\longrightarrow	& K_p/x_p \Lambda_p	\\
\downarrow	&		&\downarrow		\\
K/\Lambda	&		\stackrel{x}\longrightarrow	& K/x \Lambda.
\end{eqnarray*}

The following is by Shimura, \textit{et. al.}
\begin{theorem}[Main Theorem of Complex Multiplication, \cite{Shimura}~Chapter~5 Theorem~5.4.] 		\label{MT}
Recall that $K$ is an imaginary quadratic field. Let $\Lambda \subset K$ be a lattice in $K$, $\sigma$ be an automorphism of $\C$ invariant on $K$ (in other words, a $K$-automorphism of $\C$), $s$ be an element of $\mathbb A_K^{\times}$ so that $\sigma|_{K_{ab}}=[s, K]$, and $E$ be an elliptic curve so that there is an analytic isomorphism $\xi: \C/\Lambda \to E$. Then, there is an isomorphism $\xi': \C/ s^{-1} \Lambda \to E^{\sigma}$ so that the following is commutative:

\begin{eqnarray*}
K/\Lambda& \stackrel{\xi}\to  &  E_{tors} \\
s^{-1} \downarrow   &    &   \downarrow \sigma\\
K/ s^{-1}\Lambda  &  \stackrel{\xi'}\to   & E^{\sigma}_{tors}.
\end{eqnarray*}
($\xi'$ is uniquely determined by the above property once $\xi$ is fixed.)
\end{theorem}

\vspace{3mm}

Note that the precise definition of $s^{-1} \Lambda$ is given above.

As before, we let $P_{\tau}=(b(\tau), c(\tau))$ for $\tau \in K \cap \mathbb H$, and $\Lambda_\tau=(1, \tau)$.

For any lattice $\Lambda$ in $K$ we have the standard invariants

\[G_{2n}(\Lambda)=\sum _{\omega \in \Lambda, \omega \not=0} \omega^{-2n}, \quad g_2(\Lambda)=60\cdot G_4(\Lambda), \quad g_3(\Lambda)=140\cdot G_6 (\Lambda).\]

Suppose $E_{\tau}$ is the elliptic curve given by the Weierstrass equation 

\[y^2=4x^3-g_2(\Lambda_{\tau})x-g_3(\Lambda_{\tau})\]
so that there is an (analytic) isomorphism

\begin{eqnarray*}
\xi: \C/\Lambda_{\tau} 		& \to 		&	\quad \quad E_{\tau}		\\
z		&\mapsto		&	(\wp(z; \Lambda_{\tau}), \wp'(z; \Lambda_{\tau})).
\end{eqnarray*}

As in Theorem~\ref{MT}, $\sigma$ is any automorphism of $\C$ invariant on $K$, and $s \in \mathbb A_K^{\times}$ satisfies $[s, K]=\sigma|_{K_{ab}}$. By Theorem~\ref{MT}, there is an (analytic) isomorphism $\xi': \C/s^{-1}\Lambda_{\tau} \to E^{\sigma}_{\tau}$ such that the diagram in Theorem~\ref{MT} commutes. As well-known, there is a lattice $\Lambda'=(\omega_1', \omega_2')$ in $K$ such that

\[	g_2(\Lambda_{\tau})^{\sigma}=g_2(\Lambda'), \quad g_3(\Lambda_{\tau})^{\sigma}=g_3(\Lambda'), 	\]
and

\begin{eqnarray*}	\xi'': \C/\Lambda' 		&\to		& \quad \quad E_{\tau}^{\sigma}	\\
z		&\mapsto		& (\wp(z; \Lambda'), \wp'(z; \Lambda'))
\end{eqnarray*}
is an analytic isomorphism. Then, the composite map $\C/s^{-1} \Lambda_\tau \stackrel{\xi'}\longrightarrow E_{\tau}^{\sigma} \stackrel{\xi''^{-1}}\longrightarrow \C/\Lambda'$ is an analytic isomorphism, which is given by

\[ \C/s^{-1}\Lambda_{\tau} \stackrel{\times \lambda}\longrightarrow \C/\Lambda' 		\]
for some $\lambda \in \C^*$ (implying $\Lambda'=\lambda \cdot s^{-1} \Lambda_{\tau}$). In other words,

$$\xi'(z)=\xi''(\lambda \cdot z)$$
for $z\in \C/s^{-1} \Lambda_{\tau}$.

Therefore, for any $u \in K/\Lambda_\tau$,

\begin{eqnarray} 	\label{Houston}
\wp(u; \Lambda_{\tau})^{\sigma}		&=	&	\wp(\lambda \cdot s^{-1} u; \lambda \cdot s^{-1} \Lambda_{\tau}),\\
\label{Dallas}		\left( \wp'(u; \Lambda_{\tau}) \right)^{\sigma}		&=	&		\wp'(\lambda \cdot s^{-1} u; \lambda \cdot s^{-1} \Lambda_{\tau}).
\end{eqnarray}

Suppose $N\cdot \OO_K=\prod_{i=1}^k v_i^{n_i}$ ($n_i>0$) for some primes $v_1,\cdots, v_k$ of $\OO_K$. 

Suppose an order $\OO_c=\Z+c\OO_K$ acts on $\Lambda_\tau$ for some $c \in \Z (c>0)$ with $(c,N)=1$.

Suppose $\sigma$ is identity on $L_{\OO_c}$.  Since $\Gal(L_{\OO_c}/K)\cong \mathbb A_K^* / K^* \prod_v \OO_{c, v}^*$, $s\in \mathbb A_K^*$ satisfying $[s,K]=\left. \sigma \right|_{K_{ab}}$ should be indeed $s \in K^* \prod_v \OO_{c, v}^*$.

Write $s= \mu [\cdots, a_v, \cdots ]_v$ where $\mu \in K^*$ and $a_v \in \OO_{c,v}^*$ for each place $v$ of $K$. By the Chinese remainder theorem, there is $B \in \OO_K$ so that $B \equiv (ca_{v_i})^{-1} \pmod{v_i^{n_i}}$ for every $i=1,\cdots,k$. Let $C=cB$. Then, $C \in c\OO_K \subset \OO_c$, and naturally, $C\equiv a_{v_i}^{-1} \pmod{v_i^{n_i}}$.

Then, we have the following formula for the action of $\sigma$ on $P_\tau$.

\vspace{2mm}

\begin{proposition} 		\label{Vienna}

For an $L_{\OO_c}$-automorphism $\sigma$ of $\C$, and $C$ defined above, we have

\begin{eqnarray*}
b(\tau)^{\sigma}	&=&	\ds - \frac{\left(	\ds		\wp(  C \frac 1N;    \Lambda_{\tau})
-	\wp(  C \frac 2N;    \Lambda_{\tau})	 \right)^3}
{\ds	\wp'(  C \frac 1N;    \Lambda_{\tau})^2}	,		\\
c(\tau)^{\sigma}	&=&	\ds		-	\frac{\ds	\wp'(  C \frac 2N;    \Lambda_{\tau})^2}{\ds	\wp'(  C \frac 1N;    \Lambda_{\tau})^2}
\end{eqnarray*}

\end{proposition}

\begin{proof}
Since we assume $a_v \in \OO_{c,v}^*$ for each place $v$, $a_v^{-1}\Z_l(1,\tau)=\Z_l(1, \tau)$. Therefore, $[\cdots, a_v, \cdots]^{-1}\Lambda_\tau=\Lambda_\tau$, and $s^{-1} \Lambda_\tau=\mu^{-1} \Lambda_\tau$.

If $l$ is a prime, and $l\nmid N$, then $\frac 1N\in \Z_l$, thus $\frac 1N \equiv 0$ modulo $\Z_l(1, \tau)$. Since $a_v \in \OO_{c,v}^*$ for each $v$, $\prod_{v|l} a_v^{-1} \frac 1N \in \Lambda_\tau \otimes \Z_l$, and since $C \in \OO_c$, $C \frac 1N \in \Lambda_\tau \otimes \Z_l$. In other words, $a_v^{-1} \frac 1N \equiv C \frac 1N \equiv 0$ modulo $\Lambda_\tau \otimes \Z_l$.

If $l|N$ and $v|l$ for a prime $l$ (\textit{i.e.}, $v=v_i$ for some $i=1,\cdots, k$), by construction $C\equiv a_{v_i}^{-1} \pmod{v_i^{n_i}}$, thus $C \equiv a_{v_i}^{-1} \pmod{N\OO_{K_{v_i}}}$. In other words, $a_{v_i}^{-1} \frac 1N - C \frac 1N \in \OO_{K_{v_i}}$. Since $\OO_c \subset \Lambda_\tau$  and $\OO_{c,v_i}=\OO_{K_{v_i}}$ (because $(c,N)=1$), $a_{v_i}^{-1} \frac 1N - C \frac 1N \in \Lambda_\tau \otimes \Z_l$. In other words, $a_{v_i}^{-1} \frac 1N \equiv C \frac 1N$ modulo $\Lambda_\tau \otimes \Z_l$.

Combined we have $C\frac 1N=[\cdots, a_v, \cdots]^{-1}\frac 1N \pmod{\Lambda_{\tau}}$.

Then, by (\ref{Houston}) and (\ref{Dallas}),

\begin{eqnarray*}	
		\ds \wp(\frac 1N; \Lambda_{\tau})^{\sigma}	&=	&	\wp(\lambda \cdot s^{-1} \frac 1N; \lambda \cdot s^{-1} \Lambda_{\tau}),	\\
&=& \wp(\lambda \cdot \mu^{-1} C\frac 1N; \lambda \cdot \mu^{-1} \Lambda_{\tau}),
\end{eqnarray*}
\begin{eqnarray*}	
	\ds	\left( \wp'(\frac 1N; \Lambda_{\tau}) \right)^{\sigma}	&=	&		\wp'(\lambda \cdot s^{-1} \frac 1N; \lambda \cdot s^{-1} \Lambda_{\tau})		\\
&=& \wp'(\lambda \cdot \mu^{-1} C\frac 1N; \lambda \cdot \mu^{-1} \Lambda_{\tau}).
\end{eqnarray*}

Thus, we have

\[ b(\tau)^{\sigma}	=	\ds - \frac{\left(	\ds		 \wp(\lambda \cdot \mu^{-1} C\frac 1N; \lambda \cdot \mu^{-1} \Lambda_{\tau})
-	 \wp(\lambda \cdot \mu^{-1} C\frac 2N; \lambda \cdot \mu^{-1} \Lambda_{\tau})	 \right)^3}
{\ds	 \wp'(\lambda \cdot \mu^{-1} C\frac 1N; \lambda \cdot \mu^{-1} \Lambda_{\tau})^2}.
\]
By noting  $\wp(\lambda \mu^{-1} z; \lambda \mu^{-1} \Lambda)=(\lambda \mu^{-1})^{-2} \wp(z; \Lambda)$ and $\wp'(\lambda \mu^{-1} z, \lambda \mu^{-1} \Lambda)= (\lambda \mu^{-1})^{-3} \wp'(z; \Lambda)$, we obtain our claim.

The case for $c(\tau)^\sigma$ is similar.
\end{proof}

Furthermore, suppose $\sigma$ is identity on $L_{\OO_c, N}$ (assuming $(c,N)=1$). Since $\Gal(L_{\OO_c, N}/K)\cong \mathbb A_K^* / K^* \prod_{v \nmid N} \OO_{c, v}^* \prod_{i=1}^k 1+v_i^{n_i}$, we can choose $s=\mu [\cdots, a_v, \cdots]_v$ so that $a_{v_i}\equiv 1 \pmod{v_i^{n_i}}$ for $i=1,\cdots, k$, thus we can choose $C$ which satisfies $C \equiv 1 \pmod N$. Thus by Proposition~\ref{Vienna}, we have

\begin{eqnarray*} b(\tau)^{\sigma}=b(\tau), \quad c(\tau)^{\sigma}=c(\tau). 
\end{eqnarray*}
Therefore, $P_\tau$ is defined over $L_{\OO_c, N}$.

%%%%%%%%%%%%%%%%%%%%%%%%%
%%%%%%%%%%%%%%%%%%%%%%%%%
%%%%%%%%%%%%%%%%%%%%%%%%%
%%%%%%%%%%%%%%%%%%%%%%%%%
%%%%%%%%%%%%%%%%%%%%%%%%%
%%%%%%%%%%%%%%%%%%%%%%%%%
%%%%%%%%%%%%%%%%%%%%%%%%%

\end{section}

%%%%%%%%%%%%%%%%%%%%%%%%5
%%%%%%%%%%%%%%%%%%%%%%%%
%%%%%%%%%%%%%%%%%%%%%

%%%%%%%%%%%%%%%%%%

\begin{section}{Euler systems}

\begin{subsection}{Hecke operators}

\hfill

Recall that we let $(\alpha, \alpha')$ denote $\Z\alpha+\Z\alpha'$. Also, we will let $\Zp(\alpha, \alpha')$ denote $\Zp\cdot \alpha + \Zp \cdot \alpha'$. 
We recall that a point on $X_1(N)_{/\Q} \otimes \C$ is given by $(E, P)$ where $E$ is a generalized elliptic curve over $\C$, and $P \in E[N]$. For simplicity, let $(E, P)$ also denote the divisor $(E,P) \in Div X_1(N)$.
First, we note that for a prime $p$ ($p\nmid N$), $T_p$ acts on $Div X_1(N)$ by

\[ T_p : (E, P) \mapsto \sum _C (E/C, P') \]
where $C$ runs over all cyclic subgroups of $E[p]$ of order $p$ (there are $p+1$ such subgroups), and $P'$ is the image of $P$ under $E\to E/C$ (see \cite{Wiles-1}~page~5).

In particular, when $(E,P)=(\C/\Lambda_{\tau}, 1/N)$,

\[ T_p(E,P)=\sum_{j=0}^{p-1} ( \C/\Lambda_{\frac{\tau+j}p}, 1/N) + \C/(1/p, \tau), 1/N). \]

As in Section~\ref{Berlin}, write

\[ N=\prod_{i=1}^k v_i^{n_i}\]
for prime ideals $v_i$ of $\OO_K$ and integers $n_i>0$. We recall from Section~\ref{Field} that when $\OO$ is an order of $K$ of conductor $c$ (i.e., $\OO=\Z+c\OO_K$) with $(c,N)=1$, $L_{\OO,N}$ is the extended ring class field of level $N$ which is characterized by

\begin{eqnarray*}
\Gal(L_{\OO,N}/K) &\cong &I_K (cN)/P_{K,\Z,N}(cN) 	\\
& \cong&  \mathbb A_K^*/K^* \prod_{v \nmid N} \OO_v^* \prod_i (1+v_i^{n_i}).
\end{eqnarray*}

Let $f(z)=\sum a_n q^n$ be a (standardized) newform of level $\Gamma_1(N)$ and character $\epsilon \pmod N$, and $\mu_f: X_1(N)\to J_1(N) \to A_f$ be a modular parametrization map where the first map is given by $P \mapsto (P)-(\infty)$ for any $P \in X_1(N)$, and $A_f$ is the quotient of $J_1(N)$ given in the standard way associated to $f$.

\vspace{2mm}

\begin{notation}Suppose $P_{\tau} \in X_1(N)(L)$ for some extension $L$ of $K$. We let $P_{\tau}$ also denote the image $\mu_f(P_{\tau}) \in A_f( L)$.
\end{notation}

\vspace{2mm}

\end{subsection}

\vspace{2mm}
%%%%%%%%%%%%%%%%%%%%%

\begin{subsection}{Distribution relations}

\hfill
%%%%%%%%%%%%%%%%%%%%%

Recall that $K=\Q(\sqrt D)$ where $D$ is a square-free negative integer. Let

\[ \tau= 
\left\{ 
\begin{array}{cll}
\sqrt D	&	\text{ if }		& D \not\equiv	1 \pmod 4		\\
\ds \frac{\sqrt D+1}2		&	\text{ if }		& D \equiv	1 \pmod 4
\end{array}
\right.	\]

Suppose $c$ is a positive integer prime to $N$, and let

\[		\tau'=\frac{a+\tau}c		\]
where $a,c \in \Z$. We note that $\OO_c=\Z+c\OO_K$ acts on $(1, \tau')$. As shown in Corollary~\ref{Carp} (and also in Section~\ref{Berlin}), $P_{\tau'} \in L_{\OO_c, N}$.

\vspace{2mm}
%%%%%%%%%%%%%%%%%%%%%

\begin{subsubsection}{Proof when $p$ is prime to the conductor}

\vspace{3mm}
%%%%%%%%%%%%%%%%%%%%%

\begin{theorem}			\Label{Spain}
Suppose $p$ is a prime number, $(c,pN)=1$, $p\equiv 1 \pmod N$, and $p$ is unramified and inert over $K/\Q$. Then,

\[ \Tr_{ L_{\OO_{cp}, N} / L_{\OO_c, N} } P_{\tau'/p} = a_p(f) P_{\tau'}. \]
($P_{\tau'/p}$ can be replaced by $P_{(\tau'+j)/p}$ for any $j\in \Z$, or $P_{p\tau'}$).
\end{theorem}

\vspace{2mm}
%%%%%%%%%%%%%%%%%%%%%

\begin{proof}
Recall

\begin{eqnarray*}
T_p \left(\C/\Lambda_{\tau'}, \ds \frac 1N \right)	&=	& \sum_{j=0}^{p-1} \left(\C/\Lambda_{\frac{\tau'+j}p}, \ds \frac 1N \right) + \left(\C/( \frac 1p, \tau'), \ds \frac 1N \right) 			\\
&=&		\sum_{j=0}^{p-1} \left(\C/\Lambda_{\frac{\tau'+j}p}, \ds \frac 1N \right) + \left(\C/( 1, p \tau'), \ds \frac pN \right) 	\\
&=&		\sum_{j=0}^{p-1} \left(\C/\Lambda_{\frac{\tau'+j}p}, \ds \frac 1N \right) + \left(\C/( 1, p \tau'), \ds \frac 1N \right) 
\end{eqnarray*}
(the last equality because $p\equiv 1 \pmod N$).

For any prime $l (\not=p)$ and integer $j$,

\[ \Z_l(1, \ds \frac{\tau'+j}p ) = \Z_l (1, \tau'+j) = \Z_l (1, \tau') = \Z_l (\frac 1p, \tau').\]

On the other hand,

\begin{eqnarray*}
\Zp(1, \ds \frac{\tau'+j}p )		&=&	\Zp(1, \ds \frac{ \ds \frac{a+cj+\tau}c}p)		\\
&=& (a+cj+\tau) \cdot \Zp( \ds \frac 1{a+cj+\tau} , \frac 1{cp} )		\\
&=& (a+cj+\tau) \cdot \Zp( \ds \frac{a+cj+\bar\tau}{N_{K/\Q}(a+cj+\tau)}, \frac 1p).
\end{eqnarray*}

%%%%%%%
\textit{Case 1.} $D\not\equiv1 \pmod 4$.

First, note

\[ \ds \frac{a+cj+\bar\tau}{N_{K/\Q}(a+cj+\tau)}= \frac{a+cj-\tau}{(a+cj)^2-D}. \]
Since $p$ is unramified and inert over $K/\Q$, $(a+cj)^2-D \not\equiv 0 \pmod p$ for any $a+cj$. Thus,

\begin{eqnarray*} 
\Zp( \ds \frac{a+cj+\bar\tau}{N_{K/\Q}(a+cj+\tau)}, \frac 1p)		&=&\Zp (a+cj-\tau, \frac 1p)		\\
&=&\Zp(\tau, \frac 1p).
\end{eqnarray*}

%%%%%%
\vspace{4mm}

\textit{Case 2.} $D \equiv1 \pmod 4$.

\[ \ds \frac{a+cj+\bar\tau}{N_{K/\Q}(a+cj+\tau)}= \frac{(a+cj)+(-\tau+1)}{(a+cj)^2 + (a+cj)+\ds \frac{1-D}4}. \]
For a reason similar to the above, $(a+cj)^2 + (a+cj)+\ds \frac{1-D}4 \not\equiv 0 \pmod p$ for any $a+cj$. Thus,

\begin{eqnarray*} 
\Zp( \ds \frac{a+cj+\bar\tau}{N_{K/\Q}(a+cj+\tau)}, \frac 1p)&=&\Zp ((a+cj)+(-\tau+1), \frac 1p)		\\
&=&\Zp(\tau, \frac 1p).
\end{eqnarray*}

\vspace{5mm}

In either case,

\begin{eqnarray*}
\Zp(1, \ds \frac{\tau'+j}p )&=& (a+cj+\tau) \cdot \Zp( \ds \frac{a+cj+\bar\tau}{N_{K/\Q}(a+cj+\tau)}, \frac 1p)		\\
&=& (a+cj+\tau) \cdot \Zp(\tau, \frac 1p),
\end{eqnarray*}
and since $(c,p)=1$,

\[ \Zp(\tau, \frac 1p)=  \Zp( \ds \frac{a+\tau}c, \frac 1p)=\Zp(\tau', \frac 1p).\]

For each $j \in \Z$, let
\[ s_j=(1,1,\cdots, a+cj+\tau, 1,\cdots ) \in \mathbb A_K^*\]
where $a+cj+\tau \in \OO_{K_p}^*$ is the $p$-component. We note $s_j \in \prod_{v\nmid N} \OO_{c,v}^* \prod_i 1+v_i^{n_i}$ (because all components of $s_j$ are $1$ except for the $p$-component, and $\OO_{c,p}=\OO_{K_p}$ and $a+cj+\tau \in \OO_{K_p}^*$). So far, we have shown

\[ ( 1, \ds \frac{\tau'+j}p )= s_j \cdot ( \frac 1p, \tau'). \]

\vspace{3mm}

%%%%%%%%%%%%%%%%%%%%%%%%%%%

\begin{lemma}		\Label{Sushi}
\begin{enumerate}[(a)]
\item $s_j \ds \frac 1N=\frac 1N$ where $\ds\frac 1N$ on the left is considered an element of $K/(\frac 1p, \tau')$, and the one on the right an element of $K/ \Lambda_{\frac{\tau'+j}p}$.

\item  $\{ a+cj+\tau \}_{j=0,1,\cdots, p-1} \cup \{1\}= \OO_{c,p}^*/ \OO_{cp, p}^* (\cong \OO_{K_p}^*/(\Zp+p\OO_{K_p})^*)$. 

\item Consequently, $\{ [s_j, K]|_{ L_{\OO_{cp},N} } \} \cup \{ id \}= \Gal( L_{\OO_{cp},N} / L_{\OO_{c},N} )$.
\end{enumerate}
\end{lemma}

\vspace{3mm}

\begin{proof}
\begin{enumerate}[(a)]
\item Simply because for every $v|N$, $v$-entry of $s_j$ is 1, and for every $l\nmid N$, $\frac 1N \in \Z_l$.

\item By brute force, show the elements of the lefthand side are all distinct modulo $(\Zp+p\OO_{K, p})^*$. Since the righthand side has $p+1$ elements, we obtain our claim.

\item Because $\Gal( L_{\OO_{cp}, N} / L_{\OO_c, N}) \cong \OO_{c,p}^*/ \OO_{cp, p}^*$.

\end{enumerate}
\end{proof}

%%%%%%%%%%%%%%%%%%%%%%%%%%%

Now, choose $\sigma_j \in\operatorname{Aut}(\C)$ so that $\sigma_j|_{K_{ab}}=[s_j, K]$ for $j=0,1,\cdots, p-1$.

Recall

\[ P_{\tau'}=(b(\tau'), c(\tau'))\]
where

\[b(\tau')=-\displaystyle \frac{(\wp( \ds \frac 1N; \Lambda_{\tau'})-\wp( \ds \frac 2N; \Lambda_{\tau'}))^3}{\wp'(\ds \frac 1N; \Lambda_{\tau'})^2}, \]
\[ c(\tau')=\displaystyle - \frac{\wp'(\ds \frac2N; \Lambda_{\tau'})}{\wp'(\ds \frac 1N;  \Lambda_{\tau'})}. \]

As noted earlier,

\[ T_p P_{\tau'}= \sum_{j=0}^{p-1} P_{\frac{\tau'+j}p}+P_{p\tau'} \]
as divisors on $X_1(N)$.

Suppose $E_{p\tau'}: y^2=4x^3-g_2x-g_3$ with $g_2=g_2((\ds \frac 1p, \tau')), g_3=g_3((\ds \frac 1p, \tau'))$, and $\varphi$ be the following analytic isomorphism:

\begin{eqnarray*} \C/(\ds \frac 1p, \tau') &\to & E_{p\tau'}		\\
z &\mapsto&	\left(\wp(z; (\ds \frac 1p, \tau')), \wp'(z; (\ds \frac 1p, \tau')) \right).
\end{eqnarray*}

By Theorem~\ref{MT}, there is an analytic isomorphism $\psi: \C/s_j \cdot (\ds \frac 1p, \tau') \to E_{p\tau'}^{\sigma_j^{-1}}$ so that the following diagram is commutative:

\begin{eqnarray*}
K/(\ds \frac 1p, \tau')	& \stackrel{\varphi}\longrightarrow		& (E_{p\tau'})_{tor}		\\
\times s_j \downarrow \quad &			& \quad \downarrow \sigma_j^{-1}			\\
K/s_j \cdot (\ds \frac 1p, \tau')	& \stackrel{\psi}\longrightarrow		& (E_{p\tau'}^{\sigma_j^{-1}})_{tor}
\end{eqnarray*}

Similar to the argument in Section~\ref{Berlin}, there is some $x \in \C^*$ so that $\psi$ is the composite map

\[ \psi: \C/s_j \cdot (\ds \frac 1p, \tau') \stackrel{\times x}\longrightarrow \C/x \cdot s_j \cdot (\ds \frac 1p, \tau') \stackrel{(\wp, \wp')}  \longrightarrow   E_{p\tau'}^{\sigma_j^{-1}}.\]

Recall $s_j \cdot (\ds \frac 1p, \tau')=\Lambda_{ \frac{\tau'+j}p}$ (see the discussion before Lemma~\ref{Sushi}), and $s_j\cdot \ds \frac 1N=\frac 1N$ (Lemma~\ref{Sushi}~(a)). By substituting these into $b(\cdot)$ and $c(\cdot)$, we have

\begin{eqnarray*}			
b(\ds \frac{\tau'+j}p)		=	 -\displaystyle \frac{(\wp(s_j\cdot  \ds \frac 1N; s_j \cdot (\ds \frac 1p, \tau'))-\wp(s_j\cdot  \ds \frac 2N; s_j \cdot (\ds \frac 1p, \tau') ))^3}{\wp'(s_j\cdot \ds \frac 1N; s_j \cdot (\ds \frac 1p, \tau') )^2}
\end{eqnarray*}
which is equal to
\[			
-\displaystyle \frac{(\wp(x \cdot s_j\cdot  \ds \frac 1N; x \cdot s_j \cdot (\ds \frac 1p, \tau'))-\wp(x \cdot s_j\cdot  \ds \frac 2N; x \cdot s_j \cdot (\ds \frac 1p, \tau') ))^3}{\wp'(x \cdot s_j\cdot \ds \frac 1N; x \cdot s_j \cdot (\ds \frac 1p, \tau') )^2}
\]
because $\wp(x z; x\Lambda)=x^{-2} \wp(z;\Lambda)$, $\wp'(xz; x\Lambda)=x^{-3} \wp'(z; \Lambda)$. By the above commutative diagram, this is equal to

\begin{eqnarray*}			
\left[ -\displaystyle \frac{(\wp(\ds \frac 1N;  (\ds \frac 1p, \tau'))-\wp( \ds \frac 2N; (\ds \frac 1p, \tau') ))^3}{\wp'(\ds \frac 1N; (\ds \frac 1p, \tau') )^2}	\right]^{\sigma_j^{-1}}		
&=& \left[ -\displaystyle \frac{(\wp(\ds \frac pN;  \Lambda_{p\tau'})-\wp( \ds \frac {2p}N; \Lambda_{p\tau'} ))^3}{\wp'(\ds \frac pN; \Lambda_{p\tau'} )^2}	\right]^{\sigma_j^{-1}}		\\
&=& b(p\tau')^{\sigma_j^{-1}}.
\end{eqnarray*}
(The last equality is because $p\equiv 1\pmod N$.)

Similarly, $c(\ds \frac{\tau'+j}p)=c(p\tau')^{\sigma_j^{-1}}$.

Thus, $P_{p\tau'}^{\sigma_j^{-1}}=P_{ \frac{\tau'+j}p}$, which shows 

\begin{eqnarray}		\Label{Valencia}
 T_p(P_{\tau'})=\sum_{j=0}^{p-1} (P_{p\tau'})^{\sigma_j^{-1}}+(P_{p\tau'})=\Tr_{ L_{\OO_{cp},N} / L_{\OO_{c},N} } (P_{p\tau'})
\end{eqnarray}
(as divisors) by Lemma~\ref{Sushi} and by the assumption that $\sigma_j|_{L_{\OO_{cp},N}}= [s_j, L_{\OO_{cp},N}/K]$. Since $T_p$ acts as multiplication by $a_p(f)$ on the abelian variety $A_f$, and since $\Gal( L_{\OO_{cp},N} / L_{\OO_{c},N} )$ acts transitively on $\left\{ P_{ \frac{\tau'+j}p} \right\} \cup \{ P_{p\tau'} \}$, we obtain Theorem~\ref{Spain}.
\end{proof}

%%%%%%%%%%%%%%%%%%%%%%%%%%

\end{subsubsection}

\vspace{2mm}

\begin{subsubsection}{Proof when $p$ divides the conductor}

\vspace{3mm}

\begin{theorem}		\Label{Spanish March}
Suppose $p$ is a prime number such that $p|c$ and $p\nmid N  \operatorname{disc}(K/\Q)$. Then,

\[ \Tr_{ L_{\OO_{cp},N}/ L_{\OO_c, N} } P_{\tau'/p} = a_p(f) P_{\tau'}-\epsilon(p)P_{p\tau'}. \]
($P_{\tau'/p}$ can be replaced by $P_{(\tau'+j)/p}$ for any $j\in \Z$) where $\epsilon$ is the (Nebentypus) character of $f$. 
\end{theorem}

\vspace{2mm}

\begin{proof}

Similar to Theorem~\ref{Spain}, we have the following equality of divisors:

\begin{eqnarray*}
T_p \left(\C/\Lambda_{\tau'}, \ds \frac 1N \right)	
&=&		\sum_{j=0}^{p-1} \left(\C/\Lambda_{\frac{\tau'+j}p}, \ds \frac 1N \right) + \left(\C/( 1, p \tau'), \ds \frac pN \right) 	\\
&=&		\sum_{j=0}^{p-1} \left(\C/\Lambda_{\frac{\tau'+j}p}, \ds \frac 1N \right) + \langle p \rangle \left(\C/( 1, p \tau'), \ds \frac 1N \right) 
\end{eqnarray*}
where $\langle \cdot \rangle$ is the diamond operator (for the action of $\langle \cdot \rangle$, see \cite{Wiles-1}~Section~2).

Similar to Theorem~\ref{Spain}, for any prime $l (\not=p)$ and integer $j$,

\[ \Z_l(1, \ds \frac{\tau'+j}p )  = \Z_l (1, \tau')=\Z_l (1, \ds \frac{\tau'}p).\]

When $l=p$, we note

\begin{eqnarray*}
\Zp(1, \ds \frac{\tau'+j}p )		&=&	\Zp(1, \ds \frac{ a+cj+\tau}{cp}).
\end{eqnarray*}

Now, we have

\begin{eqnarray*}
\ds \frac{\ds \frac{\tau'+j}p}{\ds \frac{\tau'}p}		&=& \frac{(a+cj)+\tau}{a+\tau}		\\
&=& \ds \frac{[(a+cj)+\tau][a+\bar\tau]}{N_{K/\Q} (a+\tau)}.
\end{eqnarray*}

We consider $\frac{ \frac{\tau'+j}p}{ \frac{\tau'}p}$ to be an element of $K \otimes \Qp$ (and in an appropriate context, the $p$-component of the adele $\mathbb A_K^*$). In other words, if $p$ is inert, it is an element of $K_p$, and if $p$ splits completely so that $p\OO_K=\mfp\barmfp$, an element of $K_{\mfp} \times K_{\barmfp}$.
Similarly, $\Zp (1,   \frac{\tau'+j}p )$ is considered a lattice inside $K\otimes \Qp$.

\vspace{3mm}

%%%%%%%
\textit{Case 1.} $D\not\equiv1 \pmod 4$.

By simple computation,

\begin{eqnarray*} 
\ds \frac{ \frac{\tau'+j}p}{ \frac{\tau'}p}	&=& \ds \frac{[(a+cj)+\tau][a+\bar\tau]}{N_{K/\Q} (a+\tau)}		\\
&=&		\ds \frac{(a^2-D+acj)-cj\tau}{a^2-D}		\\
&=&		1+\ds \frac{ac}{a^2-D}j- \frac{c}{a^2-D}j\tau		\\
&=&		1+\ds \frac{ac}{a^2-D}j- \frac{c}{a^2-D}j (c\tau'-a)		\\
&=&		1+\ds \frac{2ac}{a^2-D}j - \frac{c^2}{a^2-D}j \tau'
\end{eqnarray*}

Noting $a^2-D \not\equiv 0 \pmod p$, it follows that

\begin{eqnarray*} 
\ds \frac{ \frac{\tau'+j}p}{ \frac{\tau'}p}	 \Zp (1, \frac{\tau'}p ) &=&	\Zp(1+\ds \frac{2ac}{a^2-D}j - \frac{c^2}{a^2-D}j \tau', \frac{\tau'+j}p )		\\
&=&\Zp(1 +\ds \frac{2ac}{a^2-D}j - \frac{c^2}{a^2-D}j \tau' + (\frac{c^2p}{a^2-D}j) \cdot (\frac{\tau'+j}p), \frac{\tau'+j}p )		\\
&=&\Zp(1 +\ds \frac{2ac}{a^2-D}j + \ds\frac{c^2 j^2}{a^2-D}, \frac{\tau'+j}p )		\\
&=&\Zp(1, \frac{\tau'+j}p )
\end{eqnarray*}
(the last line because $p|c$ and $p\nmid a^2-D$).

Noting $\ds\frac{\tau'+j}{\tau'}=1+\ds \frac{ac}{a^2-D}j- \frac{c}{a^2-D}j\tau$ as shown above, what we have shown is equivalent to

\[ (1+\ds \frac{ac}{a^2-D}j- \frac{c}{a^2-D}j\tau) \Zp (1, \frac{\tau'}p )= \Zp(1, \frac{\tau'+j}p ).\]
\vspace{3mm}

%%%%%%%

\textit{Case 2.} $D \equiv1 \pmod 4$.

Note
\[ N_{K/\Q} (a+\tau)=a^2+a+\ds \frac{1-D}4 \not\equiv 0 \pmod p. \]

\begin{eqnarray*} 
[(a+cj)+\tau][a+\bar\tau]&=& a(a+cj)+\ds \frac{1-D}4 + a\tau+(a+cj)\bar\tau	\\
&=&a^2+acj+\ds \frac{1-D}4 +a\tau+(a+cj)(1-\tau)		\\
&=&	(a^2+acj+\ds \frac{1-D}4+a+cj)-cj\tau
\end{eqnarray*}

So, similar to Case 1,

\begin{eqnarray*} 
\ds \frac{ \frac{\tau'+j}p}{ \frac{\tau'}p} &=& \frac{(a^2+acj+ \frac{1-D}4+a+cj)-cj\tau}{a^2+a+ \frac{1-D}4}	\\
&=&	1+ \ds \frac {acj+cj}{a^2+a+ \frac{1-D}4} - \frac{cj}{a^2+a+ \frac{1-D}4} \tau		\\
&=&	1+ \ds \frac {acj+cj}{a^2+a+ \frac{1-D}4} - \frac{cj}{a^2+a+ \frac{1-D}4} (c\tau'-a)		\\
&=&		1+ \ds \frac {2acj+cj}{a^2+a+ \frac{1-D}4} - \frac{c^2j}{a^2+a+ \frac{1-D}4} \tau'	,	\\
\end{eqnarray*}
thus

\begin{eqnarray*} 
\ds \frac{{\tau'+j}}{{\tau'}}	\cdot \Zp (1, \frac{\tau'}p ) &=&	\Zp(1+ \ds \frac {2acj+cj}{a^2+a+ \frac{1-D}4} - \frac{c^2j}{a^2+a+ \frac{1-D}4} \tau', \frac{\tau'+j}p )		\\
&=&\Zp(1+\ds \frac{2acj+cj+c^2j^2}{a^2+a+ \frac{1-D}4}, \frac{\tau'+j}p )		\\
&=&\Zp(1, \frac{\tau'+j}p )
\end{eqnarray*}
(the last line because $\ds \frac{2acj+cj+c^2j^2}{a^2+a+\ds \frac{1-D}4} \equiv 0 \pmod p$).

Noting $\ds \frac{{\tau'+j}}{{\tau'}}=1+ \ds \frac {acj+cj}{a^2+a+ \frac{1-D}4} - \frac{cj}{a^2+a+ \frac{1-D}4} \tau$ as shown above, what we have shown is equivalent to

\begin{eqnarray*}
\left( 1+ \ds \frac {acj+cj}{a^2+a+ \frac{1-D}4} - \frac{cj}{a^2+a+ \frac{1-D}4} \tau \right) \cdot \Zp (1, \frac{\tau'}p ) 		= \Zp(1, \frac{\tau'+j}p ) 
\end{eqnarray*}
for each $j \in \Z$.

Let $s_j=(1,1,\cdots, \ds \frac{\tau'+j}{\tau'}, 1,\cdots ) \in \mathbb A_K^*$ for $j=0,1,\cdots, p-1$ where $\frac{\tau'+j}{\tau'}$ is the $p$-component. We have shown that in both Case 1 and Case 2,

\[ s_j \Lambda_{\frac{\tau'}p}=\Lambda_{\frac{\tau'+j}p}. \]

We note that in both Case 1 and Case 2, 

\[ \frac{\tau'+j}{\tau'} \in 1+c(\OO_K\otimes \Zp) \subset (\OO_c \otimes \Zp)^*, \]
which implies $s_j \in \prod_{v\nmid N} \OO_{c,v}^* \prod_i 1+v_i^{n_i}$.

\begin{lemma}

\begin{enumerate}[(a)]
\item $s_j \ds \frac 1N=\frac 1N$ where $\ds \frac 1N$ on the left is in $K/\Lambda_{\frac{\tau'}p}$, and the one on the right in $K/\Lambda_{\frac{\tau'+j}p}$. 

\item  $\{ \ds \frac{\tau'+j}{\tau'} \}_{j=0,1,\cdots, p-1} \cong ((\OO_c)\otimes \Zp)^*/ ((\OO_{cp})\otimes \Zp)^*$. 

\item Consequently, $\{ [s_j, K]|_{L_{\OO_{cp},N}} \}_{j=0,1,\cdots, p-1} = \Gal( L_{\OO_{cp},N} / L_{\OO_{c},N} )$.
\end{enumerate}
\end{lemma}

\vspace{3mm}

\begin{proof}
Similar to Lemma~\ref{Sushi} except for (b). As for (b), first suppose $p^n \parallel c$. 
We show that for any $A,B\in \Zp$ ($(B,p)=1$), $1+Ap^nj+Bp^nj\tau$ for $j=0,1,\cdots, p-1$ are all distinct modulo $\left( \OO_{p^{n+1}} \otimes \Zp \right)^*$, and thus $\{ 1+Ap^n j+Bp^n j\tau \}_{j=0,1,\cdots, p-1}=(\OO_{p^n} \otimes \Zp)^*/(\OO_{p^{n+1}} \otimes \Zp)^*$ by noting that both sides have the same number of elements.
\end{proof}

%%%%%%%%%%%%%%%%%%%%%%%%%%%

Now, choose $\sigma_j \in\operatorname{Aut}(\C)$ so that $\sigma_j|_{K_{ab}}=[s_j, K]$ for $j=0,1,\cdots, p-1$.

Similar to Theorem~\ref{Spain}, 

\begin{eqnarray*}			
b(\ds \frac{\tau'+j}p)	&=&	 b(\ds \frac{\tau'}p)^{\sigma_j^{-1}}		\\
c(\ds \frac{\tau'+j}p) &=&c(\ds \frac{\tau'}p)^{\sigma_j^{-1}}
\end{eqnarray*}

Thus, $P_{ \frac{\tau'}p}^{\sigma_j^{-1}}=P_{ \frac{\tau'+j}p}$, which shows 

\[ T_p(P_{\tau'})=\sum_{j=0}^{p-1} (P_{\frac{\tau'}p})^{\sigma_j^{-1}}+\langle p \rangle (P_{p\tau'})=\Tr_{ L_{\OO_{cp},N} / L_{\OO_{c},N}} (P_{\frac{\tau'}p})+ \langle p \rangle (P_{p\tau'}).\]
Since $T_p$ acts as multiplication by $a_p(f)$ on the abelian variety $A_f$, $\langle p \rangle$ as multiplication by $\epsilon(p)$, and $\Gal( L_{\OO_{cp},N} / L_{\OO_{c},N} )$ acts transitively on $\{ P_{ \frac{\tau'+j}p} \}_{j=0,\cdots, p-1}$, we obtain Theorem~\ref{Spanish March}.
\end{proof}

\end{subsubsection}

\end{subsection}

%%%%%%%%%%%%%%%%%%%%%%%
%%%%%%%%%%%%%%%%%%%%%%%%%

\begin{subsection}{Congruence Relations}		\Label{Congruence Relations}

\hfill

Where $\mathcal L$ is a local field of residue characteristic $p$ which is prime to $N$, $\nu$ is the maximal ideal of $\OO_{\mathcal L}$, $\mathbb F$ is $\OO_{\mathcal L}/\nu$, and $\mathcal A$ is the N\'eron model of $A_f$ over $\Zp$ (which exists because $(p,N)=1$), we note that there is the standard reduction map $\operatorname{red}_\nu: A_f(\mathcal L) \to \mathcal A_{/\mathbb F_p}(\mathbb F)$. Also let $J_{/\mathbb F_p}$ denote the fiber of the N\'eron model of $J_1(N)_{/\Q}$ over $\Zp$ (which exists because 
$(p,N)=1$).

Recall that $\tau=\sqrt D$ if $D \not\equiv 1 \pmod 4$, and $(\sqrt D+1)/2$ if $D\equiv 1 \pmod 4$, and $\tau'=\ds \frac{a+\tau}c$ with $(c,N)=1$.

\vspace{2mm}

\begin{theorem}		\Label{Congruence}
Suppose $p$ is a prime that is inert over $K/\Q$, $p \equiv 1 \pmod N$, and $(p, c)=1$. Let $\lambda$ be any prime of $L_{\OO_c, N}$ lying above $p$, and $\lambda'$ be any prime of $L_{\OO_{cp}, N}$ lying above $\lambda$. Then, we have

\[ (p+1)\operatorname{red}_{\lambda'} P_{\tau'/p}= (\operatorname{Frob}_{p}+p \cdot \epsilon(p) \cdot \operatorname{Frob}_{p}^{-1}) \operatorname{red}_{\lambda} P_{\tau'} = a_p(f) \operatorname{red}_{\lambda} P_{\tau'}. \]

\end{theorem}

\vspace{2mm}

\begin{remark}

As we will see in the proof, $P_{\tau'/p}$ can be replaced by $P_{(\tau'+j)/p}$ for any $j \in \Z$, or $P_{p\tau'}$.

We state Theorem~\ref{Congruence} with $\operatorname{Frob}_{p}+p \cdot \epsilon(p) \cdot \operatorname{Frob}_{p}^{-1}$ to make it appear similar to Kolyvagin's congruence relation ((\ref{Kolyvagin-3}) in Section~\ref{Introduction}), but in practice, the statement with $a_p(f) \operatorname{red}_{\lambda} P_{\tau'}$ might be more suitable.

\end{remark}

\vspace{2mm}

\begin{proof}
As we see in the proof of Theorem~\ref{Spain} (see (\ref{Valencia})),

\[
T_p(P_{\tau'})=\Tr_{ L_{\OO_{cp},N} / L_{\OO_{c},N} } (P_{p\tau'}).
\]
As noted in Theorem~\ref{Spain}, $P_{p\tau'}$ can be replaced by $P_{(\tau'+j)/p}$ for any $j \in \Z$ because they are transitive under the action of $\Gal(L_{\OO_{cp},N} / L_{\OO_{c},N})$. To be consistent with the notation in the theorem, we replace $P_{p\tau'}$ with $P_{\tau'/p}$.

Our condition implies that $(p)$ splits completely over $L_{\OO_c, N}/K$, and $\lambda'$ is totally ramified over $L_{\OO_{cp}, N}/L_{\OO_c, N}$ (therefore $\OO_{L_{\OO_{cp}, N}}/\lambda' \cong \OO_{L_{\OO_c, N}}/\lambda \cong \OO_K/(p)$). Thus, every $\sigma \in \Gal(L_{\OO_{cp}, N}/L_{\OO_c, N})$ is the identity modulo $\lambda'$. Hence, 

\[ \operatorname{red}_{\lambda} T_p(P_{\tau'})=(p+1) \operatorname{red}_{\lambda'} P_{\tau'/p}.	\]

On the other hand, by the Eichler-Shimura relation,

\[ T_p = \operatorname{Frob}_p+\langle p \rangle \operatorname{Ver}_p \]
as action on $J_{/\mathbb F_p}$. Here $\operatorname{Frob}_p$ is the Frobenius and $\operatorname{Ver}_p$ is the Verschiebung of the group scheme $J_{/\mathbb F_p}$ (\cite{Wiles-1}~Section~4, (4.1)). We recall that $\langle p \rangle$ acts as multiplication by $\epsilon(p)$ as an action on $\mathcal A_{/\mathbb F_p}$.

By recalling that $\operatorname{Ver}_p$ is characterized by $\operatorname{Frob}_p \operatorname{Ver}_p=\operatorname{Ver}_p \operatorname{Frob}_p=p$, we obtain

\[ (p+1)\operatorname{red}_{\lambda'} P_{\tau'/p}= (\operatorname{Frob}_{p}+p \cdot \epsilon(p) \cdot \operatorname{Frob}_{p}^{-1}) \operatorname{red}_{\lambda} P_{\tau'}. \]
(We may consider $\operatorname{Frob}_p$ as an automorphism of finite fields.) Thus, we obtain the first part of our equation.

Then, note the standard result $\operatorname{Frob}_{p}+p \cdot \epsilon(p) \cdot \operatorname{Frob}_{p}^{-1}=a_p(f)$ on $\mathcal A_{/\mathbb F_p}$. (Or, we can simply argue that $T_p$ acts as multiplication by $a_p$ as an action on $\mathcal A_{/\mathbb F_p}$.) Thus, we also obtain the second part of the equation.
\end{proof}

\vspace{2mm}

In his original work, Kolyvagin does not make any use of the second distribution relation (\ref{Kolyvagin-2}) (see Section~\ref{Introduction}), which in fact does not seem to appear in his work. This relation is what Rubin calls the distribution relation in the $p$-direction. Instead, he uses congruence relations.

Rubin points out (\cite{Rubin-2}, and \cite{Rubin-3}~Remark~2.1.5, and also Section~4.8) that the congruence relations are often unnecessary, or can be derived from other distribution relations if ``the Euler system satisfies distribution relations in the $p$-direction'' (such as (\ref{JKK-2}) in Section~\ref{Introduction}, which is Theorem~\ref{Spanish March}).

On the other hand, suppose one insists on using a congruence relation. Theorem~\ref{Congruence} has a more restrictive condition that $p$ is inert over $K/\Q$ and $p\equiv 1 \pmod N$ (whereas Kolyvagin's congruence condition applies to all primes prime to the discriminant of $K/\Q$). However, a careful reading of \cite{Rubin-0}~Section~4 (which is another work that uses the congruence conditions) indicates that one does not need a congruence relation for every prime, and in fact, it is probably enough that it holds for inert primes. Adding the condition $p \equiv 1$ is probably fine, too.

Either way, we can claim that $\{ P_{\frac{a+\tau}c} \}$ satisfies enough relations to be an Euler system.

\end{subsection}

\end{section}

\begin{section}{Appendix}				\label{Texas}

Our description of the action of the Galois groups on $P_{\tau}$ in Section~\ref{Berlin} is not necessarily efficient for practical computation. In this section, we apply Koo, Shin, and Yoon's work (\cite{Koo-Shin-Yoon}~Section~4, \cite{Koo-Yoon}~Section~2) on the action of the Galois groups of class field extensions on the singular values of modular functions to obtain formulas which may be more readily useful in practice.

Following \cite{Koo-Shin-Yoon}, when $K$ is an imaginary quadratic field of discriminant $d_K$, let

\[ \theta=\left\{		
\begin{array}{lll}
\sqrt{d_K}/2	&	\text{if }	&d_K \equiv 0 \pmod 4	\\
(-1+\sqrt{d_K})/2	&	\text{if }	&d_K \equiv 1 \pmod 4.
\end{array}
\right.
\]
We let $\mathcal F_N$ be the extension of the function field $\Q(j(\tau))$ generated by the Fricke functions indexed by $r \in \frac 1N \Z^2/\Z^2$ (see \cite{Koo-Shin-Yoon}~Section~4). By the theory of modular functions, it is known that $\mathcal F_N$ is the set of all functions in $\C(X(N))$ whose Fourier coefficients are in $\Q(\zeta_N)$, $\mathcal F_1$ is simply $\Q(j(\tau))$, and

\[ \Gal(\mathcal F_N/\mathcal F_1) \cong \operatorname{GL}_2(\Z/N\Z) / \{ \pm I_2 \} \cong G_N \cdot \operatorname{SL}_2(\Z/N\Z) / \{ \pm I_2 \} \]
where

\[ G_N = \left\{ \left. \begin{bmatrix} 1&0 \\ 0&d \end{bmatrix} \; \right| \; d \in (\Z/N\Z)^* \right\}.	\]

We note $b(\tau), c(\tau) \in \mathcal F_N$ (\cite{Jeon-Kim-Lee}). 

The action of $\Gal(\mathcal F_N/\mathcal F_1)$ on $\mathcal F_N$ can be made explicit through $\operatorname{GL}_2(\Z/N\Z) / \{ \pm I_2 \} \cong G_N \cdot \operatorname{SL}_2(\Z/N\Z) / \{ \pm I_2 \}$ as follows (\cite{Koo-Yoon}~Section~2): $\begin{bmatrix} 1&0 \\ 0&d \end{bmatrix} \in G_N$ acts on $\mathcal F_N$ by

\[ \sum_{n>>-\infty} c_n q_{\tau}^{\frac nN} \mapsto  \sum_{n>>-\infty} c_n^{\sigma_d} q_{\tau}^{\frac nN}		\]
where $\sigma_d \in \Gal(\Q(\zeta_N)/\Q)$ is given by $\zeta_N^{\sigma_d}=\zeta_N^d$. 

And, $\gamma \in  \operatorname{SL}_2(\Z/N\Z) / \{ \pm I_2 \}$ acts on $h \in \mathcal F_N$ by

\[ h^{\gamma}(\tau)=h(\tilde \gamma \tau) \]
where $\tilde \gamma \in \operatorname{SL}_2(\Z)$ is a lifting of $\gamma$.

Following the notation of \cite{Koo-Shin-Yoon}, we let $H$ and $K_{(N)}$ be the Hilbert class field, and the ray class field modulo $N\OO_K$ of $K$ respectively. (Note that $K_{(N)}$ is $L_{\OO_K, N}$ in our earlier notation.) The action of $\Gal(K_{(N)}/H)$ is given explicitly through the singular values of $\mathcal F_N$ as follows (\cite{Koo-Shin-Yoon}~Section~2): Let $\operatorname{min}(\theta, \Q)=x^2+B_{\theta}x+C_{\theta} \in \Z[x]$, and let

\[ W_{N, \theta}=\left\{ \begin{bmatrix} t-B_{\theta}s & -C_{\theta}s \\ s&t \end{bmatrix} \in \operatorname{GL}_2 (\Z/N\Z) \; : \; t,s \in \Z/N\Z \right\}. \]

By \cite{Koo-Shin-Yoon}~Proposition~4.1, we can identify $K_{(N)}$ with the field

\[ K(h(\theta)\: | \: h \in \mathcal F_N \text{ is defined and finite at }\theta).		\]
Koo, Shin, and Yoon identify $\Gal(K_{(N)}/H)$ with the image of $W_{N, \theta}$ by the following surjection:

\begin{proposition}[\cite{Koo-Shin-Yoon}~Proposition~4.2] \label{4.2}
\begin{eqnarray*} W_{N,\theta} &\to& \Gal(K_{(N)}/H))		\\
\alpha&\mapsto&		\left( h(\theta) \mapsto h^{\alpha}(\theta) \right)
\end{eqnarray*}
\end{proposition}
If $d_K \leq -7$, then its kernel is $\{ \pm I_2 \}$. By applying this to $b(\tau)$ and $c(\tau)$, we can compute the action of $\Gal(K(N)/H)$ on $P_\tau$. (The computational advantage is that we only need to know  $W_{N, \theta}$.)

The action of $\Gal(H/K)$ (and by extension, the action of $\Gal(K_{(N)}/K)$) is given as follows (\cite{Koo-Shin-Yoon}~Proposition~4.3). Let $C(d_K)$ the form class group of discriminant $d_K$, which is the set of primitive positive definite quadratic forms $aX^2+bXY+cY^2 \in \Z[X,Y]$ under a proper equivalence relation. (See \cite{Koo-Shin-Yoon}~p.351 for an explicit classification of the elements of $C(d_K)$.) We note that $C(d_K)\cong \Gal(H/K)$ (\cite{Cox}~Theorem~7.7).

For a reduced quadratic form $Q=aX^2+bXY+cY^2$ of $\operatorname{disc}Q=d_K$, we set

\[ \theta_Q=(-b+\sqrt{d_K})/2a, \]
and set $\beta_Q=(\beta_p)_p \in \prod_p \operatorname{GL}_2(\Zp)$ (where $p$ runs over all primes) for $\beta_p$ given in \cite{Koo-Shin-Yoon}~p.351-352.

\begin{proposition}[\cite{Koo-Shin-Yoon}~Proposition~4.3] 	\label{4.3} Assume $d_K \leq -7$ and $N\geq 1$. (Note that by this assumption, $W_{N,\theta}/\{ \pm I_2\}  \cong \Gal(K_{(N)}/H)$.) Then, we have a bijective map

\begin{eqnarray*}
W_{N,\theta}/\{ \pm I_2\} \times C(d_K)		&\to & \Gal(K_{(N)}/K) \\
(\alpha, Q)	&\mapsto	& \left( h(\theta) \mapsto h^{\alpha \beta_Q}(\theta_Q) \right)
\end{eqnarray*}
\end{proposition} 
As \cite{Koo-Shin-Yoon} notes, there is $\beta \in \operatorname{GL}_2^+(\Q) \cap \operatorname{M}_2(\Z)$ with $\beta \equiv \beta_p \pmod{N\Zp}$ for all primes $p$ dividing $N$, and the action of $\beta_Q$ is understood to be that of $\beta$.

The bijection in Proposition~\ref{4.3} is not a group isomorphism (but the map in Proposition~\ref{4.2} is a group homomorphism), but from the computational perspective, it should not matter.

This bijection gives a computational description of the action of $\Gal(K_{(N)}/K)$ on our functions $b(\tau), c(\tau)$ when $\tau=\theta$. \cite{Koo-Shin-Yoon} does not explicitly say much when $\tau$ is not $\theta$, but its authors told us that their result can be generalized to any $\tau \in K$. In such a case, the action will be that of $\Gal(L_{\OO_n, N}/K)$ where $\OO_n$ is the order of $K$ acting on $\Lambda_\tau$, which is compatible with our work.

\end{section}

%%%%%%%%%%%%%%%%%%%%%%%%5

\end{document}